\documentclass[12pt,a4paper]{article}
\usepackage[utf8]{inputenc}
\usepackage{amsmath,amssymb,amsthm}
\usepackage[margin=1in]{geometry}
\usepackage{hyperref}
\usepackage[dvipsnames]{xcolor}
\usepackage{comment}

\theoremstyle{plain}
\newtheorem{theorem}{Theorem}[section]
\newtheorem{lemma}[theorem]{Lemma}
\newtheorem{proposition}[theorem]{Proposition}

\newtheorem{conjecture}[theorem]{Conjecture}
\theoremstyle{definition}
\newtheorem{definition}[theorem]{Definition}
\theoremstyle{remark}
\newtheorem{remark}[theorem]{Remark}

\newcommand{\defeq}{\stackrel{\textnormal{def}}{=}}

\newenvironment{salign}{
    \begin{equation}
    \begin{aligned}
}{
    \end{aligned}
    \end{equation}
    \ignorespacesafterend
}

\newcommand{\Z}{\mathbb{Z}}
\newcommand{\Q}{\mathbb{Q}}

\newcommand{\R}{\mathbb{R}}

\renewcommand{\P}{\mathbb{P}}

\DeclareMathOperator{\Ht}{Ht}

\title{Lattice Point Visibility Along Powers of Polynomials}

\author{
  Abraham Lobsenz \\
  \and
  Tristan Phillips \\
}

\date{}

\begin{document}

\sloppy

\maketitle

\begin{abstract}
We study lattice point visibility along polynomial lines of sight and give a new proof of the Visibility Density Conjecture of Chaubey and Pandey for a large class of polynomials.
\end{abstract}

%% ===================================================================
\section{Introduction}\label{sec:intro}
%% ===================================================================

Lattice point visibility is a classical problem in number theory.
A point $(a,b)$ in the integer lattice $\Z\times \Z$ is said to be \textit{visible from the origin} if the line segment between the points $(0,0)$ and $(a,b)$ contains no additional lattice points. The lattice point visibility problem asks: What proportion of lattice points are visible from the origin?

Since $(a,b)$ being visible from the origin is equivalent to $\gcd(a,b)=1$,  the lattice point visibility problem is equivalent to asking for the probability that two randomly chosen integers are coprime. This reformulation relates the problem to the Basel problem, posed by Mengoli in 1650, which asks for the precise value of the sum
\begin{equation}
    \zeta(2) = \sum_{n=1}^\infty \frac{1}{n^2}.
\end{equation}
Euler's celebrated 1734 solution of this problem showed that
\begin{equation}
    \zeta(2) = \prod_{p \text{ prime}} \left( 1 - \frac{1}{p^2}\right)^{-1} =\frac{\pi^2}{6}.
\end{equation}
Since the probability that two integers are not both divisible by a given prime $p$ is $1-1/p^2$, this shows that the probability that two integers are coprime is $6/\pi^2$, and hence the density of visible lattice points is also $6/\pi^2$. Subsequent work  justifying and extending Euler's argument was made by Weierstrass, Mertens, Dirichlet, Ces\`{a}ro, and Sylvester.
For more modern treatments of the classic lattice point visibility problem, see \cite[Theorem 332]{HardyWright} and \cite[\S 3.2]{Apostol}.

In this paper we consider lattice point visibility along polynomial lines of sight. Let $F(x)\in \Z[x]$ be a polynomial with positive leading coefficient. 
A lattice point $(a,h)\in \Z_{>0}\times \Z_{>0}$ is \emph{visible along $F(x)$} if there exists a rational number $t\in \Q$ such that $h=t\, F(a)$ and $a$ is the smallest positive integer $u$ such that $t\, F(u)$ is a positive integer; otherwise $(a,h)$ is said to be \textit{invisible along $F(x)$}.
This notion recovers classic lattice point visibility when $F(x)=x$, and generalizes the notions of polynomial lattice point visibility considered in \cite{GoinsHarrisKubikMbirika, ChaubeyPandey}, in which polynomials were  assumed to pass through the origin. 

Define the density of visible lattice points as
\begin{equation}
    D(F) \defeq \lim_{N\to\infty} \frac{\#\{(a,b)\in \Z_{>0}^2 : \max\{a,b\}\leq N \text{ and  $(a,b)$ is visible along $F(x)$}\}}{N^2}.
\end{equation}
For the classic lattice point visibility problem, one has $D(x)=1/\zeta(2)=6/\pi^2$.  Goins, Harris, Kubik, and Mbirika \cite{GoinsHarrisKubikMbirika} extended this to monomial lines of sight, proving that $D(\alpha x^b)=1/\zeta(1+b)$. 

For polynomials with more than one distinct root a striking change in behavior is expected: \textit{almost all} lattice points are expected to be visible. We formalize this expectation in the following conjecture, which extends a conjecture of Chaubey and Pandey \cite[Conjecture 1.6]{ChaubeyPandey} in the case of polynomials passing through the origin:

\begin{conjecture}[Visibility Density Conjecture]\label{conj:visibility_conjecture}
    If $F(x)\in \Z[x]$ is a polynomial with at least two distinct roots, then the set of lattice points in $\Z_{>0}\times \Z_{>0}$ that are visible along $F(x)$ has density $1$, i.e., 
    \begin{equation}
         D(F)= 1.
    \end{equation}
\end{conjecture}

Let $a,b,u,v\in \Z$ with $au\neq 0$ and $b>0$. Then the polynomials excluded from Conjecture \ref{conj:visibility_conjecture} are precisely the nonzero constant polynomials and the polynomials of the form
\begin{equation}
F(x)=a(ux+v)^b.
\end{equation}
By possibly rescaling, we may assume $\gcd(u,v)=1$. Making the linear change of variables $ux+v\mapsto X$, lattice point visibility along $F(x)$ becomes equivalent to lattice point visibility along $aX^b$ with the lattice 
\begin{equation}\label{eq:lattice_with_local_condition}
    \{(c,d)\in \Z_{>0}\times \Z_{>0} : c\equiv v \pmod{u}\}.
\end{equation}
A slight modification of \cite{GoinsHarrisKubikMbirika} then shows
\begin{equation}
    D(a(ux+v)^b) = \prod_{p\nmid u} \left(1-\frac{1}{p^{1+b}}\right) = \frac{1}{\zeta(1+b)}\prod_{p|u} \left(1 - \frac{1}{p^{1+b}}\right)^{-1}.
\end{equation}
Alternatively, this can be derived as a special case of a result of the second author \cite[Theorem 4.0.5]{Phillips}.\footnote{See Remark \ref{rem:batyrev-manin} for further discussion.}

\begin{comment}
Indeed, choose a prime $q\equiv 1 \pmod{|u|}$. If $q\mid (ux+v)$, then
\[
x'=\frac{(ux+v)/q-v}{u}\in \Z,
\]
and for all sufficiently large $x$ we have $0<x'<x$. Thus, if in addition $q^b\mid h$, then
\[
\frac{h}{F(x)}\,F(x')=\frac{h}{q^b}\in \Z_{>0},
\]
so $(x',h/q^b)$ blocks $(x,h)$ along $F$. Hence a positive proportion of lattice points are invisible, and therefore any limiting density for such an $F$ must be strictly less than $1$.

However, a linear change of variable does \emph{not} in general preserve the exact value of the density. For example,
\[
D(2x+1)=\prod_{p\neq 2}\left(1-\frac{1}{p^2}\right)=\frac{8}{\pi^2}<1,
\]
whereas
\[
D(x)=\frac{6}{\pi^2}.
\].
\end{comment}

In \cite{ChaubeyPandey}, Chaubey and Pandey make progress towards Conjecture \ref{conj:visibility_conjecture} by establishing asymptotic lower bounds for $D(F)$ when $F(x)\in \Z[x]$ is a polynomial with coefficients that are coprime positive integers and $x$ divides $F(x)$. They make this lower bound explicit in the case $F(x)=x^2+mx$ with $m\geq 1$. Recently, Ahuja \cite{Ahuja} further explored polynomial lattice point visibility, obtaining an exact formula for the number of visible lattice points using inclusion-exclusion and other related results.

After the present work was completed and posted, we were kindly informed by Chaubey that Conjecture \ref{conj:visibility_conjecture}, in the case of polynomials passing through the origin, has been resolved by Chaubey, Pandey, and Regavim \cite{ChaubeyPandeyRegavim}, who prove that $D(F)=1$ for every such polynomial that is not of the form $\alpha x^b$. Their work, which was carried out independently of and was not available to us during the preparation of this article, settles the original conjecture \cite[Conjecture 1.6]{ChaubeyPandey} in full. In the range where our results overlap, our methods yield a stronger quantitative bound (see Remark \ref{rem:CPR_comparison}).

The main result of this paper verifies Conjecture \ref{conj:visibility_conjecture} for a large class of polynomials:
        
\begin{theorem}\label{thm:main}
    Let $f(x)\in \Z[x]$ be a polynomial of degree $d\geq 2$ with positive leading coefficient and having at least two distinct roots over an algebraic closure $\overline{\Q}$ of $\Q$, and let $m\in \Z_{\geq 2}$. 
    Then Conjecture \ref{conj:visibility_conjecture} is true for $F(x)=f(x)^m$.
\end{theorem}

More precisely, we prove the following quantitative upper bound for the number of invisible lattice points for the polynomials in Theorem \ref{thm:main}:

\begin{theorem}\label{thm:main2}
    Let $F(x)$ be as in Theorem \ref{thm:main}. 
    For each $r,s\in \Z_{>0}$   with $s>r$, define the polynomial
    \begin{equation}
        G_{s,r}(x,y)\defeq s\, f(y) - r\, f(x).
    \end{equation}
    Then there is an integer  $\delta_F\geq 2$ such that $ G_{s,r}(x,y)$ has no irreducible factors of degree less than $\delta_F$ over $\R$ for all $r$ and $s$, and the number of invisible lattice points along $F(x)$ is $O_{F,\varepsilon}(N^{1+1/\delta_F+\varepsilon})$.
\end{theorem}

\begin{remark}
    Note that if Conjecture \ref{conj:visibility_conjecture} holds for a polynomial $F(x)$, then it holds for any nonzero multiple of $F(x)$. In particular, if the conjecture holds for $F(x)$ it will also hold for $-F(x)$. Thus it would suffice to prove Conjecture \ref{conj:visibility_conjecture} for $F(x)$ with positive leading coefficient. For this reason, the condition that $f(x)$ (and hence $F(x)$) has positive leading coefficient in Theorem \ref{thm:main} is not a significant restriction.
\end{remark}

\begin{remark}\label{rem:batyrev-manin}
    Interpreting the classical lattice point visibility problem in terms of lines of rational slope through the origin, it turns out that the problem is closely connected to the problem of counting rational points of bounded height on the projective line $\P^1$. Indeed, if $\Ht$ denotes the usual Weil height on $\P^1(\Q)$, then
    \begin{equation}
        \#\{[a:b]\in \P^1(\Q) : \Ht([a:b])\leq N\} \sim\frac{2}{\zeta(2)} N^2.
    \end{equation}
    Here, the extra factor of $2$ appears since lines of \textit{negative} rational slope are also counted. Similarly, the lattice point visibility problem along the polynomials $\alpha x^b$ is equivalent to counting rational points of bounded height on the weighted projective lines $\P(1,b)$. A special case of a result of Deng \cite{Deng} shows that
       \begin{equation}
        \#\{[a:b]\in \P(1,b)(\Q) : \Ht([a:b])\leq N\} \sim
        \begin{cases}
            \frac{2}{\zeta(1+b)}N^2 & \text{ if } 2\nmid b,\\
            \frac{4}{\zeta(1+b)}N^2 & \text{ if } 2|b, 
        \end{cases}
    \end{equation}
    predating \cite{GoinsHarrisKubikMbirika} by nearly two decades. 
    Similarly, results of Benedetti, Estupi\~n\'an, and Harris~\cite{BenedettiEstupinanHarris} on higher dimensional lattice point visibility can be viewed as counting rational points of bounded height in weighted projective spaces $\P(b_1,\dots,b_n)$, which is again a special case of work of Deng \cite{Deng}.

    In \cite{Phillips} the second author proves a result on counting rational points of bounded height on weighted projective stacks which satisfy local conditions. Such results allow one to count visible points in subsets of lattices with congruence conditions, such as the lattice \eqref{eq:lattice_with_local_condition}.
    
    More generally, the Batyrev--Manin conjecture \cite{FrankeManinTschinkel, BatyrevManin, Peyre}, concerning the number of rational points of bounded height on Fano varieties, can be thought of as counting visible lattice points in certain regions. 
    In contrast to much of the previous work on lattice point visibility, the problems we consider are not equivalent to counting rational points of bounded height on some projective variety in the same way since $y-F(x)$ will be neither homogeneous nor weighted homogeneous. 
\end{remark}

%TODO: discuss methods and how $S_F(N)$ sums arise, and our bound for them. (optional) 
%TODO: definitely mention used of Pila-Bombieriri, `real determinant method'

\begin{remark}\label{rem:CPR_comparison}
    In the cases where $F(0)=0$, Theorem \ref{thm:main2} improves upon results in \cite{ChaubeyPandeyRegavim}, which establish that the number of invisible lattice points is $O_{F,\varepsilon}(N^{2-1/(n^2-1)+\varepsilon})$ in general, and $O_{F,\varepsilon}(N^{2-1/(2n-1)+\varepsilon})$ when $F$ is separable.
\end{remark}

\begin{remark}
    In forthcoming work, using different methods, we will improve the bounds in Theorem \ref{thm:main2} in the case $d=2$, showing that the number of invisible lattice points is
    \begin{equation}
        \begin{cases}
            O_{F,\varepsilon}(N^{1+\varepsilon}) & \text{ if } m=2,\\
            O_{F}(N\log(N)) & \text{ if } m\ge 3.
        \end{cases}
    \end{equation}
\end{remark}

\subsection*{Acknowledgements}
We would like to thank Asher Auel and Akash Singha Roy for interesting conversations related to this project.
We would also like to thank Sneha Chaubey, Ashish Kumar Pandey, and Shvo Regavim for sharing a preprint of their paper \cite{ChaubeyPandeyRegavim} with us.
TP was supported by the National Science Foundation, via grant DMS-2303011.

%% ===================================================================
\section{Reducing the visibility problem to a GCD sum}\label{sec:visibility}
%% ===================================================================

In this section we cover preliminaries on visibility along polynomials, give a sufficient condition for proving the Visibility Density Conjecture in terms of the growth rate of a certain double sum (Proposition~\ref{prop:S_F(N)_to_density_conjecture}), and rewrite this sum in a useful way.

Let $F(x)\in \Z[x]$ be a polynomial with positive leading coefficient. 

\begin{definition}[Visibility along~$F$]\label{def:visibility}
     A lattice point $(a,h)\in \Z_{>0}\times \Z_{>0}$ is \emph{visible along $F(x)$} if there exists a rational number $t\in \Q$ such that $h=t\, F(a)$ and $a$ is the smallest positive integer $u$ such that $t\, F(u)$ is a positive integer.
    A lattice point is called \emph{invisible along $F(x)$} if it is not visible along $F(x)$.
    Finally, we say that a lattice point $(b,k)$ \emph{blocks $(a,h)$ from being visible along $F(x)$} if $b<a$ and there exists a $t\in \Q$ such that $h=t\,F(a)$ and $k=t\, F(b)$.
\end{definition}

For integers $M,N$, define the set $[M,N]\defeq\{n\in \Z : M\le n\le N\}$, which is empty if $M>N$.
Define the set of visible lattice points along $F$ by
\begin{equation}\label{eq:visible-set}
    \mathrm{Visible}_F(N) \defeq \{(a,h)\in [1,N]^2 : \text{$(a,h)$ is visible along } F\}
\end{equation}
and the set of invisible lattice points along $F$ by
\begin{equation}\label{eq:invisible-set}
    \mathrm{Invisible}_F(N) \defeq \{(a,h)\in [1,N]^2 :  \text{$(a,h)$ is invisible along } F\}. 
\end{equation}
We write $\#\mathrm{Visible}_F(N)$ and $\#\mathrm{Invisible}_F(N)$ for
the number of visible and invisible lattice points, respectively. 
%Often the subscript will be dropped when the polynomial $F$ is clear from context.

Because $F(x)$ has positive leading coefficient, there exists a positive integer $n_F$ such that if $a>n_F$ and $1\le b<a$, then
%\begin{equation}
     $F(b)<F(a)$ and $F(a)>0$.
%\end{equation}

Define $\mathrm{Block}_F(a,b;N)$ to be the set 
\begin{equation}
    \{(a,h)\in [1,N]^2 : \exists\ k\in\Z \text{ such that $(b,k)$ blocks $(a,h)$ from being visible along $F$}\}.
\end{equation}
Note that when $N>n_F$, we have $k\in [1,N]$.

\begin{comment}}
Define $E_F(N)$ to be the set of invisible lattice points in $[1,N]^2$ blocked only by lattice points $(b,k)$ with $b\leq n_F$,
\begin{equation}
    E_F(N) \defeq \{ (a,h)\in [1,N]^2 : \text { if } (a,h)\in \mathrm{Block}_F(a,b;N), \text{ then } b\leq n_F\}.
\end{equation}
\end{comment}

Define $E_F(N)$ to be the set of invisible lattice points in $[1,N]^2$ whose blockers all have first coordinate at most $n_F$,
\begin{salign}
    E_F(N) &\defeq \{ (a,h)\in \mathrm{Invisible}_F(N) : \text { if } (a,h)\in \mathrm{Block}_F(a,b;N), \text{ then } b\leq n_F\}\\
    &= \sum_{a\leq N}\sum_{1\leq b\leq n_F} \#\mathrm{Block}_F(a,b;N).
\end{salign}

%An invisible lattice point $(a,h)\in [1,N]^2$ either has $a\leq n_F$, or else $a>n_F$ and $(a,h)\in \mathrm{Block}_F(a,b;N)$ for some $b$ with $1\leq b<a$. Separating out the blockers with $b\leq n_F$, we obtain
Then we have
\begin{equation}\label{eq:block_bound}
    \#\mathrm{Invisible}_F(N)\leq \sum_{a\leq N}\sum_{n_F< b<a} \#\mathrm{Block}_F(a,b;N) + \#E_F(N).
\end{equation}
%where $E_F(N)$ counts the invisible lattice points $(a,h)\in [1,N]^2$  whose only blockers have $b\leq n_F$. 
Note that there can be multiple lattice points blocking a given invisible lattice point, and therefore the inequality \eqref{eq:block_bound} need not be an equality.

\begin{lemma}\label{lem:E_f(N)_bound}
    We have $\# E_F(N)\ll_F N$.     
\end{lemma}

\begin{proof}
Partition the points $(a,h)\in E_F(N)$ according to whether $a\le n_F$ or $a>n_F$.
\begin{itemize}
\item If $a\le n_F$, then $(a,h)\in [1,n_F]\times [1,N]$, so there are at most $n_FN$ such points.
\item If $(a,h)\in E_F(N)$ with $a>n_F$, then $(a,h)$ is invisible, so it has a visible blocker $(b,k)$. By the definition of $E_F(N)$, every blocker of $(a,h)$ satisfies $b\le n_F$.
On the other hand, each visible lattice point $(b,k)$ with $b\leq n_F$ can block at most one point in $E_F(N)\cap [n_F+1, N]^2$, for if it blocked two distinct points, one would block the other, contradicting their simultaneous inclusion in $E_F(N)$. Therefore, the number of points in $E_F(N)$ with $a > n_F$ is bounded by the number of visible lattice points with $b \leq n_F$. For $N$ sufficiently large (e.g., $N>n_F$) this quantity is less than $n_F N$.  
\end{itemize}
Combining these bounds when $N>n_F$ yields $\# E_F(N) < 2n_F N = O_F(N)$.
\end{proof}

We now bound the size of the sets $\mathrm{Block}_F(a,b;N)$.

\begin{lemma}\label{lem:block_upper_bound}
    If $a,b\in \Z_{>0}$ are integers satisfying $n_F<b<a\leq N$, then
    \begin{equation}\label{eq:block_upper_bound}
        \#\mathrm{Block}_F(a,b; N)\leq N\, \frac{ \gcd(F(a), F(b))}{F(a)}.
    \end{equation}
\end{lemma}

\begin{proof}
If $(a,h)\in \mathrm{Block}_F(a,b;N)$, then there exists a $t\in \Q$ and $k\in[1,N]$ such that
\[
h=t\,F(a)\qquad \text{and} \qquad k=t\,F(b).
\]
This implies
\[
\frac{h}{F(a)}=t=\frac{k}{F(b)},
\]
and therefore
\[
h\,F(b)=k\,F(a).
\]
Dividing by the greatest common divisor of $F(a)$ and $F(b)$, we obtain
\begin{equation}
    k \frac{F(a)}{\gcd(F(a),F(b))} = h \frac{F(b)}{\gcd(F(a),F(b))}.
\end{equation}
Since $n_F<b<a$, we have $F(a),F(b)>0$. Hence
\[
\frac{F(a)}{\gcd(F(a),F(b))}
\qquad \text{and} \qquad
\frac{F(b)}{\gcd(F(a),F(b))}
\]
are coprime positive integers. It follows that $F(a)/\gcd(F(a),F(b))$ divides $h$, so there exists a $\lambda\in [1,N]$ such that
\begin{equation}
    h=\lambda \frac{F(a)}{\gcd(F(a), F(b))}.
\end{equation}
But we must also have $h\leq N$, so
\begin{equation}
    \lambda = h\,\frac{\gcd(F(a), F(b))}{F(a)}\leq N\, \frac{ \gcd(F(a), F(b))}{F(a)}.
\end{equation}
Since, for fixed $a$ and $b$, there is one choice of $h$ for each choice of $\lambda$, this implies the lemma.
\end{proof}

Define the double sum
\begin{equation}
    S_F(N)\defeq \sum_{a\leq N}\sum_{n_F< b<a}  \frac{ \gcd(F(a), F(b))}{F(a)}.
\end{equation}
By the inequality \eqref{eq:block_bound},  Lemma \ref{lem:E_f(N)_bound}, and Lemma \ref{lem:block_upper_bound}, we have
\begin{equation}\label{eq:invisible_S_F(N)_bound}
    \#\mathrm{Invisible}_F(N)\leq N\, S_F(N) + O_F(N).
\end{equation}
Therefore the density of visible lattice points is 
\begin{salign}
    D(F) 
    &= \lim_{N\to\infty}\frac{\#\mathrm{Visible}_F(N)}{N^2}\\
    &= 1 - \lim_{N\to\infty}\frac{\#\mathrm{Invisible}_F(N)}{N^2}\\
    &\geq 1 - \lim_{N\to\infty} \frac{N\, S_F(N) + O_F(N)}{N^2} \\
    &= 1 - \lim_{N\to\infty} \frac{S_F(N)}{N}.
\end{salign}
Thus we have the following sufficient condition for proving Conjecture \ref{conj:visibility_conjecture}:

\begin{proposition}\label{prop:S_F(N)_to_density_conjecture}
If $S_F(N) = o(N)$, then the Visibility Density Conjecture (Conjecture \ref{conj:visibility_conjecture}) is true for $F(x)$.
\end{proposition}

The remainder of the paper is devoted to proving $S_F(N) = o(N)$ for the polynomials $F(x)$ in Theorem~\ref{thm:main}.
%% ===================================================================
%\section{Rearranging the the sum $S_F(N)$}\label{sec:rearranging}
%% ===================================================================
Our first step towards this goal is to rewrite $S_F(N)$.

Henceforth, let $f(x)\in \Z[x]$ be a polynomial of degree $d\geq 2$ with positive leading coefficient and having at least two distinct roots over $\overline{\Q}$, let $m\in \Z_{\geq 2}$, and let $F(x)=f(x)^m$. 

Since $f$ has positive leading coefficient, we may choose $n_f\in \Z_{\geq 0}$ such that if $a>n_f$ and $1\le b<a$, then $\max\{0, f(b)\}<f(a)$.
Since $f(a)>f(b)$ implies $f(a)^m>f(b)^m$, we may take $n_F$ to equal $n_f$.

By the definition of $n_f$, whenever $a>b>n_f=n_F$ we have
\[
f(a)>f(b)>0.
\]
Define positive integers
\begin{equation}\label{eq:gqr-def}
g(a,b) \defeq \gcd\bigl(f(a),\, f(b)\bigr), \quad
s(a,b) \defeq \frac{f(a)}{g(a,b)}, \quad \text{ and } \quad  r(a,b) \defeq \frac{f(b)}{g(a,b)}.
\end{equation}

Note that $\gcd(s(a,b), r(a,b)) = 1$ and $s(a,b) > r(a,b) \geq 1$. %, so $s \geq 2$.
Since $F(x)=f(x)^m$, we have
\begin{salign}
    \gcd(F(a),F(b)) 
 &= \gcd(f(a)^m,\ f(b)^m) \\
 &= \gcd((g(a,b)\ s(a,b))^m,\ (g(a,b)\ r(a,b))^m) \\
 &= g(a,b)^m \gcd(s(a,b)^m,\ r(a,b)^m) \\
 &= g(a,b)^m.
\end{salign}
Therefore
\begin{equation}\label{eq:gcd-ratio}
\frac{\gcd(F(a), F(b))}{F(a)}
  = \frac{g(a,b)^m}{f(a)^m}
  = \frac{1}{s(a,b)^m}.
\end{equation}
For each $s,r\in \Z_{>0}$ with $s>r$ define the counting function
\begin{equation}
    M_{s,r}(N) \defeq \#\{ (a,b) \in [n_F+1, N]^2 : s(a,b)=s\ \text{ and }\ r(a,b)=r\}. 
\end{equation}
Note that if $a\leq N$, then $s(a,b)\leq f(a)\leq f(N)\ll_f N^d$, and therefore, if $s> f(N)$, then $M_{s,r}(N)=0$. 

The sum $S_F(N)$ from Section \ref{sec:visibility} can be written in terms of the counting functions $M_{s,r}(N)$ as follows:
\begin{equation}\label{eq:S-rearranged}
 S_F(N) 
  = \sum_{n_F < b < a \leq N}  \frac{1}{s(a,b)^m}
  = \sum_{2 \leq s \leq f(N) }\;\sum_{r=1}^{s-1}
      \frac{M_{s,r}(N)}{s^m}.
\end{equation}

%% ===================================================================
\section{Density one for polynomials of arbitrary degree}\label{sec:general}
%% ===================================================================

In this section we prove Theorem~\ref{thm:main2}.
To do this we bound the counting function $M_{s,r}(N)$ by the number of integer solutions to certain polynomial equations (Proposition \ref{prop:Msr-upper_bound}) and then apply a bound of Pila for the number of integer solutions to polynomials.

%\subsection{An upper bound for $M_{s,r}(N)$}

\begin{proposition}\label{prop:Msr-upper_bound}
    For each $s,r\in \Z_{>0}$ with $s>r$, define the polynomial
    \begin{equation}
        G_{s,r}(x,y) \defeq s f(y) - r f(x).
    \end{equation}
    Then
    \begin{equation}
        M_{s,r}(N) \leq \#\{ (a,b)\in [1,N]^2 : \ G_{s,r}(a,b)=0\}.
    \end{equation}
\end{proposition}

\begin{proof}
    We show that if $(a,b)$ is counted by $M_{s,r}(N)$, then $G_{s,r}(a,b)=0$.
    If $(a,b)$ is counted by $M_{s,r}(N)$, then 
    \begin{equation}
        \frac{f(a)}{\gcd(f(a),f(b))} = s
        \qquad \text{ and } \qquad
        \frac{f(b)}{\gcd(f(a),f(b))} = r.
    \end{equation}
    Multiplying both equalities by $g(a,b)=\gcd(f(a),f(b))$ gives
    \begin{equation}
        f(a) = s\, g(a,b)\quad \text{ and } \quad f(b) = r\, g(a,b).
    \end{equation}
    Hence $s\,f(b) = srg(a,b) = r\,f(a)$, so that $G_{s,r}(a,b)=0$.
\end{proof}

\noindent Our strategy will be to apply the following bound of Pila, which built off of his work with Bombieri on the real determinant method \cite{BombieriPila1989}. 

\begin{theorem}[{\cite[Theorem~2]{Pila1995}}]\label{thm:pila}
If $G(x,y)\in \R[x,y]$ is an irreducible polynomial of
degree $\delta \geq 2$,
then for every $\varepsilon > 0$,
\begin{equation}
\#\left\{(a,b) \in [1,N]^2 :  G(a,b)=0 \right\}
  \ll_{\delta, \varepsilon} N^{1/\delta + \varepsilon},
\end{equation}
where the implied constant  depends only on $\delta$
and~$\varepsilon$, and not on the specific polynomial~$G$.
\end{theorem}

\begin{comment}
\textcolor{red}{A $\overline{\mathbb{Q}}$-irreducible factor of a $\mathbb{Q}$-polynomial can have non-real coefficients e.g., $y^2 + x^2 \in \mathbb{Q}[x,y]$ factors as $(y-ix)(y+ix)$ over $\overline{\mathbb{Q}}$, and neither factor lies in $\mathbb{R}[x,y]$. So Pila's theorem as stated (for $\mathbb{R}[x,y]$) does not apply to those factors directly. We could: factor $G_{s,r}$ first over $\mathbb{Q}$ (which equals factoring over $\mathbb{R}$, by Gauss), apply Pila to each $\mathbb{Q}$-irreducible factor $Q_j$, and note that $\deg(Q_j)$ is a sum of equal degrees of its $\overline{\mathbb{Q}}$-conjugate pieces and hence $\geq \delta_F$.}

\textcolor{blue}{TP: Factoring over $\Q$ is not the same as factoring over $\R$. For example $x^2-2=(x-\sqrt{2})(x+\sqrt{2})$. But you are right that one cannot apply Theorem 3.2 to a factorization over $\overline{\Q}$. Thanks for catching this! I have rewritten things so that Theorem 3.2 is applied to  a factorization over $\R$. I believe this avoids the analysis of $\overline{\Q}$-conjugate pieces of the factorization, but let me know if I am mistaken. }
\end{comment}

However, the polynomials $G_{s,r}(x,y)$ need not be irreducible. We proceed by applying Theorem \ref{thm:pila} to each irreducible factor of $G_{s,r}(x,y)$ over $\R$.
%, noting that irreducibility over $\overline{\Q}$ implies irreducibility over $\R$. 
Checking the hypotheses of Theorem \ref{thm:pila} for $G_{s,r}(x,y)$, we show that it has no linear factor over $\overline{\Q}$, which implies that it has no linear factor over $\R$.

\begin{proposition}[No linear factor]\label{prop:no-linear}
    Let $f \in \Z[x]$ be a polynomial of degree $d \geq 2$ with at least two distinct roots
    in $\overline{\Q}$, and let $r,s\in \Z_{\geq 1}$ be positive integers with $r<s$. 
    Then  the polynomial
    \begin{equation}
         G_{s,r}(x,y) = s\,f(y) - r\,f(x) \in \Q[x,y]
    \end{equation}
    has no linear factor over $\overline{\Q}$.
\end{proposition}

\begin{proof}
    Suppose for contradiction that $G_{s,r}(x,y)$ admits a linear factor over $\overline{\Q}$. Then there exist $\alpha,\beta,\gamma\in \overline{\Q}$, not all zero, such that
    \begin{equation}
        L(x,y) = \alpha x + \beta y + \gamma
    \end{equation}
    divides $G_{s,r}(x,y)$. 

    First we consider the case $\alpha=0$. In this case $L(x,y)=\beta y +\gamma$, so the substitution $y=-\gamma/\beta$ gives
    \begin{equation}
        s\ f(-\gamma/\beta) - r\ f(x) = 0. 
    \end{equation}
    Therefore $f(x)$ must equal the constant $s\ f(-\gamma/\beta)/r$, contradicting the assumption $\deg(f)\geq 2$. 
     A similar argument shows that $\beta=0$ is impossible.

    We now consider the case $\alpha\beta\neq 0$. In this case $x = -(\beta y +\gamma)/\alpha$. Substituting this into $G_{s,r}(x,y)$, we obtain the polynomial identity
    \begin{equation}\label{eq:polynomial_identity}
        s\ f(y) = r\ f\left(-\frac{\beta y + \gamma}{\alpha}\right). 
    \end{equation}
    %Let $c_d$ denote the leading coefficient of $f(x)$. Then, by comparing leading coefficients of \eqref{eq:polynomial_identity}, we have
    %\begin{equation}
    %    s c_d = r c_d (-\beta/\alpha)^d.
    %\end{equation}
    
    Consider the map
    \begin{salign}
        T:\overline{\Q} &\to \overline{\Q}\\
            z & \mapsto -\frac{\beta z +\gamma}{\alpha}.
    \end{salign}
    Since $\alpha\neq 0$ this map is well defined, and since $\beta\neq 0$ it is injective.\footnote{In fact, $\beta\neq 0$ implies that the map is a bijection, although we will only need injectivity to prove the proposition.}
    
    Let 
    \begin{equation}
        R_f\defeq \{\rho\in \overline{\Q} : f(\rho)=0\}
    \end{equation}
    denote the set of roots of $f(x)$. From the polynomial identity \eqref{eq:polynomial_identity}, for each root $\rho\in R_f$, we have
    \begin{equation}
        0 = s\ f(\rho) = r\ f(T(\rho)).
    \end{equation}
    Since $T$ is injective, this implies that $T$ permutes the (finite) set of roots $R_f$ (i.e., $T$ corresponds to an element in the permutation group of the set $R_f$). Therefore, there exists a positive integer $k\in \Z_{>0}$, such that if we compose $T$ with itself $k$ times, then it will correspond to the identity map on $R_f$. By the assumption that $f$ has at least two distinct roots, this implies that the map $T^k:\overline{\Q}\to\overline{\Q}$ fixes at least $\# R_f\geq 2$ points. But any map of the form $z \mapsto az+b$ that fixes two distinct points must be the identity map (i.e., $a=1$ and $b=0$). This implies $T^k(z)=z$. From this, and iterating the identity \eqref{eq:polynomial_identity}, we have
    \begin{salign}
        r^k f(z)
        &= r^k f(T^k(z)) \\
        &= r^{k-1}\left( r f(T(T^{k-1}(z))) \right)\\
        &= r^{k-1} s f(T^{k-1}(z))\\
        &= r^{k-2} s \left( f(T(T^{k-2}(z)))\right)\\
        &= r^{k-2} s^2 f(T^{k-2}(z)) \\
        &= \cdots\\
        &= s^k f(z).
    \end{salign}
    Therefore $r^k=s^k$, contradicting the assumption $r<s$. 
\end{proof}

%\subsection{Uniform pair count}

\begin{lemma}[Uniform pair count]\label{lem:pair-count-hd}
%Under the standing hypotheses, 
Let $\delta_F\in \Z_{>0}$ be the largest integer such that for all $s,r\in \Z_{>0}$ with $s>r$, the polynomial $G_{s,r}(x,y)$ has no irreducible factors over $\R$ of degree less than $\delta_F$.
For every pair of integers $s > r \geq 1$ and every $\varepsilon > 0$, we have
\begin{equation}
M_{s,r}(N) \ll_{d,\varepsilon} N^{1/\delta_F + \varepsilon}.
\end{equation}
%where the implied constant depends only on $d$ and~$\varepsilon$.
\end{lemma}

\begin{proof}
By Proposition \ref{prop:Msr-upper_bound} we have
    \begin{equation}\label{eq:msr-bound}
        M_{s,r}(N) \leq \#\{ (a,b)\in [1,N]^2 :  G_{s,r}(a,b)=0\}.
    \end{equation}
Let $G_{s,r}(x,y)=g_1(x,y)\cdots g_t(x,y)$ be a factorization of $G_{s,r}(x,y)$ into irreducible polynomials $g_i(x,y)\in \R[x,y]$. 
%By Proposition \ref{prop:no-linear} the degree of each $g_i(x,y)$ is at least $2$. 
Applying Theorem \ref{thm:pila} to each $g_i(x,y)$, we have
\begin{salign}\label{eq:pila-bound}
    & \#\{ (a,b)\in [1,N]^2 :  G_{s,r}(a,b)=0\}\\
    & \qquad \leq \sum_{i=1}^t \#\{ (a,b)\in [1,N]^2 :  g_{i}(a,b)=0\} \\
    & \qquad \ll_{d, \varepsilon} \sum_{i=1}^t N^{1/\deg(g_i)+\varepsilon} \\
    & \qquad \ll_{d,\varepsilon} N^{1/\delta_F+\varepsilon}.
\end{salign}
Combining the bounds \eqref{eq:msr-bound} and \eqref{eq:pila-bound} proves the lemma. 
\end{proof}

\begin{proof}[Proof of Theorem \ref{thm:main2}.]
By equation~\eqref{eq:S-rearranged} and Lemma~\ref{lem:pair-count-hd},
\begin{equation}  \label{eq:S-hd-bound}
S_F(N)
  \ll_{d,\varepsilon}
    \sum_{s=2}^{f(N)}\;\sum_{r=1}^{s-1}
    \frac{N^{1/\delta_F+\varepsilon}}{s^m}
  \ll_{d,\varepsilon}
    N^{1/\delta_F+\varepsilon}
    \sum_{s = 2}^{f(N)} \frac{1}{s^{m-1}}.
\end{equation}
If $m>2$ then $m-1\geq 2$, and therefore
\begin{equation}\label{eq:m>2_sum_bound}
     \sum_{s = 2}^{f(N)} \frac{1}{s^{m-1}} \leq \sum_{s = 2}^{\infty} \frac{1}{s^{m-1}} \ll_{m} 1.
\end{equation}
If $m=2$ then $m-1=1$, and for any $\varepsilon>0$ we have
\begin{equation}\label{eq:m=2_sum_bound}
     \sum_{s = 2}^{f(N)} \frac{1}{s} \leq 1 + \log(f(N)) \ll_{f} 1 + d\log(N) \ll_{f, \varepsilon} 1+N^{\varepsilon}.
\end{equation}
Combining the bounds \eqref{eq:S-hd-bound}, \eqref{eq:m>2_sum_bound}, and \eqref{eq:m=2_sum_bound}, for any $\varepsilon>0$ we have
\begin{equation}
    S_F(N) \ll_{F,\varepsilon} 
        N^{1/\delta_F+\varepsilon}.  
\end{equation}
This, together with the inequality \eqref{eq:invisible_S_F(N)_bound} and Proposition \ref{prop:no-linear}, implies Theorem \ref{thm:main2}.
\end{proof}

\bibliographystyle{alpha}
\bibliography{references}

\end{document}